\documentclass[12pt]{article}
\usepackage{amsmath,amsthm}
\usepackage{amsfonts,amsmath,comment}
\usepackage{amssymb}
\usepackage{fullpage}
\usepackage{epsf}
\usepackage{color}
\usepackage{graphicx}
\usepackage{graphics}
\usepackage{pdflscape}
\usepackage{rotating}
\usepackage{rotating}

\newtheorem{theorem}{Theorem}[section]
\newtheorem{lemma}[theorem]{Lemma}

\newtheorem{corollary}[theorem]{Corollary}

\theoremstyle{definition}

\newtheorem{example}[theorem]{Example}

\newtheorem{remark}[theorem]{Remark}

\renewcommand{\geq}{\geqslant}
\renewcommand{\leq}{\leqslant}

\def \md#1{{\,({\rm mod}\ #1)}}

\title{
Biembeddings of cycle systems using integer Heffter arrays}
\author{
Nicholas J. Cavenagh\thanks{Department of Mathematics, The University of Waikato, Private Bag 3105, Hamilton 3240, New Zealand.
\texttt{nickc@waikato.ac.nz}}\and
Diane M. Donovan\thanks{School of Mathematics and Physics, The University of Queensland,
 Queensland 4072,
Australia. (\texttt{dmd@maths.uq.edu.au})}\and
Emine \c{S}. Yazici\thanks{Department of Mathematics, Ko\c{c} University, Sar{\i}yer,
34450, \.{I}stanbul, Turkey (\texttt{eyazici@ku.edu.tr})}
}

\date{August 2019}

\begin{document}

\maketitle

\begin{abstract}
In this paper we will show the existence of a face $2$-colourable biembedding of the complete graph onto an orientable surface where each face is a cycle of a fixed length $k$, for infinitely many values of $k$.  
In particular, under certain conditions, we show that there exists at least $(n-2)[(p-2)!]^2/(e^2 kn)$ non-isomorphic face $2$-colourable biembeddings of $K_{2nk+1}$ in which all faces are cycles of length $k=4p+3$.
These conditions are: $n\equiv 1\md{4}$, $k\equiv 3\md{4}$ and either $n$ is prime or $n\gg k$ and $n\equiv 0\md{3}$ implies $p\equiv 1\md{3}$. To achieve this result we begin by verifying the existence of  $(n-2)[(p-2)!/e]^2$ non-equivalent Heffter arrays,  $H(n;k)$, which satisfy the conditions:  (1) for each row and each column the sequential partial sums determined by the natural ordering must be distinct modulo $2nk+1$; (2) the
composition of the natural orderings of the rows and columns is equivalent to a single cycle permutation on the entries in the array. The existence of Heffter arrays $H(n;k)$ that satisfy condition (1) was established earlier in \cite{BCDY} and in this current paper we vary this construction and show that there are at least $(n-2)[(p-2)!/e]^2$ such non-equivalent $H(n;k)$ that satisfy condition (1) and then show that each of these Heffter arrays also satisfy condition (2) under certain conditions.
\end{abstract}

\section{Introduction}

A $k$-cycle system of order $v$ is an edge disjoint decomposition of the complete
graph $K_v$ into cycles of length $k$. Cycle systems can be represented as  embeddings of the underlying graph  on a surface (or pseudosurface), where the cycles correspond to faces in the embedding.
Researchers have exploited this connection 
 in the study of the ``Heawood Map Colouring Conjecture'', see \cite{R,GG} and related problems.
In this paper we focus on decompositions of $K_v$ into cycles of constant length, however in general the underlying graph need not be the complete graph and the cycles need not be of constant length.
When the embedding is a proper face $2$-colourable
embedding of the complete graph $K_v$, in which each face corresponds to a cycle of length $k$
and each colour class corresponds to a $k$-cycle system,  we say the pair of $k$-cycle systems
{\em biembeds} in the surface.

It is clear that not all pairs of $k$-cycle systems  biembed on a surface, however it not obvious which pairs of such  systems  biembed.
 Further, to date, comparatively little is known about the spectrum of values of $k$ for which there exists a pair of biembeddable $k$-cycle systems.

When $k=3$, a $k$-cycle system  is commonly referred to as a Steiner
triple systems of order $v$ or STS($v$).
For these systems, it is known that necessary and sufficient conditions for the existence of  biembeddings of pairs of STS($v$) are: (a) $v\equiv 1, 3\md{6}$  and $v\geq 9$ for non-orientable surfaces; and (b) $v\equiv  3, 7\md{12}$ for orientable surfaces, \cite{GK,R}. These systems have been studied extensively and the early survey by Grannell and Griggs \cite{GG} is an excellent starting point for further information. The reader may also refer to the recent work by Korzhik \cite{K}.

Ellingham and Stephens \cite{ES} show that 
for odd $v\geq 7$ there exists a pair of biembeddable Hamilton cycle systems (i.e. $k=v$) of order $v$. 
In 2016 Griggs and McCourt \cite{GM} developed new constructions for biembeddings of symmetric ($k=(v -1)/2$) $k$-cycle decompositions of $K_{v}$ and showed necessary and sufficient conditions for the existence of a biembedding of symmetric $k$-cycle systems on a non-orientable surface if $k\geq 4$  and on an orientable surface if $k$  is odd and $k\geq 3$.
Other studies connecting  cycle decompositions and embeddings on orientable and non-orientable surfaces include
 \cite{A,
 CMPP,DM,GG,GrM,GM,M}
 and  \cite{ItalE}.

In 2015  Archdeacon \cite{A} studied biembeddings of cycle systems of the complete graph on a surface and formalized the connection between biembeddings  and Heffter arrays.
 Heffter arrays arise as an extension to
 Heffter's \cite{H} famous first difference problem:
 partition the set $\{1,\dots,3m\}$ into $m$ triples $\{a,b,c\}$ such that either $a+b=c$ or $a+b+c$ is divisible by $6m+1$. This problem was solved by Peltesohn, \cite{P}, some 40 years later, for all $m\neq 3$, a result that also implies the existence of  cyclic Steiner triple systems on the given order; see \cite{CD}. A natural extension to Heffter's first difference problem  is: can we identify a set of $m$ subsets $\{x_1,\dots,x_s\}\subset \{-ms,\dots, -1,1,\dots,ms\}$ such that the sum of the entries in each subset is divisible by $2ms+1$ and further if $x$ occurs in one of the subsets, $-x$ does not occur in any of the subsets? We call the set of $m$ such subsets a {\em Heffter system}. Two Heffter systems, $H_1=\{H_{11},\dots, H_{1m}\}$, $|H_{1i}|=s$ for $i=1,\dots m$, and  $H_2=\{H_{21},\dots, H_{2n}\}$, $|H_{2j}|=t$ for $j=1,\dots n$,  where $sm=nt$, are said to be {\em orthogonal} if for all $i,j$, $|H_{1i}\cap H_{2j}|\leq 1$.

Given the connection between Heffter's first difference problem and biembeddings of pairs of $3$-cycle systems (STS$(v)$), one may ask which Heffter systems yield biembeddings of cycle decompositions, where the length of the cycles may vary.

To study this problem we follow the work of Archdeacon, \cite{A}, and Dinitz and Mattern, \cite{DM}, who observed that  an orthogonal Heffter system is equivalent to a
 {\em Heffter array} $H(m,n;s,t)$ which is an $m\times n$ array of integers such that:
\begin{itemize}
\item each row contains $s$ filled cells and each column contains $t$ filled cells;
\item the elements in every row and column sum to $0$ in ${\mathbb Z}_{2ms+1}$; and
\item for each integer $1\leq x\leq ms$, either $x$ or $-x$ appears in the array.
\end{itemize}
The {\em support} of an array  is taken to be the set of absolute values of the entries occurring in the array.
In \cite{ABD} it was shown that
a $H(m,n;n,m)$ exists for all possible values of $m$ and $n$.

If $m=n$ and necessarily $s=t=k$, we say the Heffter array is {\em square}, usually denoted by $H(n;k)$. If the elements in every row and column sum to $0$ in ${\mathbb Z}$, we refer to the array as an {\em integer} Heffter array.
The spectrum for square Heffter arrays has been completely determined as stated in the following theorem, see \cite{ADDY, DW,CDDY}.

\begin{theorem} {\rm \cite{ADDY, DW,CDDY}} \label{SquHefSpec}
There exists an $H(n; k)$ if and only if $3 \leq k \leq  n$. Also there exists an integer $H(n;k)$ if and only if $nk\equiv 0,3 \md{4}$.
\label{mainthm}
\end{theorem}

Archdeacon \cite{A} went on to prove that a Heffter array $H(m,n;s,t)$ that admits a simple and compatible ordering of the rows and columns, can be used to construct a face $2$-colourable biembedding of the complete graph
$K_{2ms+1}$ on an orientable surface. The definitions of simple and compatible are as follows.

Given a row $r$ of a Heffter array $H(m,n;s,t)$, if there exists a cyclic ordering $\phi_r=(a_0,a_1,\dots ,a_{s-1})$ of the entries of row $r$ such that, for $i=0,\dots, s-1$,
the partial sums
$$\alpha_i=\displaystyle\sum_{j=0}^{i} a_j \md{2ms+1}$$
 are all distinct, we say that $\phi_r$ is {\em simple}. A simple ordering of the entries of a column may be defined similarly.
If every row and column of a Heffter array $H(m,n;s,t)$ has a simple ordering, we say that the array is {\em simple}.

Suppose that a simple cyclic ordering  $\phi_r=(a_1,a_2,\dots ,a_s)$ of a row $r$ of a Heffter array has the property that whenever entry $a_i$ lies in cell $(r,c)$ and entry $a_{i+1}$ lies in cell
$(r,c^{\prime})$, then $c<c^{\prime}$. That is, the ordering for the row $r$ is taken from  left to right across the array. Observe that if this ordering is simple then the ordering from right to left is also simple and vice versa. We say that $\phi_r$ is the {\em natural} ordering for the rows and define a natural column ordering in a similar way with the ordering going from top to bottom.
If the natural ordering for every row and column is also a simple ordering, we say that the Heffter array is {\em globally simple}.

The composition of the cycles $\phi_r$,  for  each row $r\in [m]$,  is a permutation, denoted here  $\omega_r$, on the entries of the Heffter array. Similarly we may define the permutation
$\omega_c$ as the composition of the cycles $\phi_c$, for the columns $c\in [n]$.
If, the permutation $\omega_r\circ \omega_c$ can be written as a single cycle of length $ms=nt$, we say that
$\omega_r$ and $\omega_c$ are  {\em compatible} orderings for the Heffter array.

An abelian group $(H,+)$ (with generator labelled $1$) has a {\em sharply vertex-transitive} action on an embedded graph ${\mathcal G}$ if the vertices of
${\mathcal G}$ are labelled with the elements of $H$ and the permutation $x\rightarrow x+1$, when applied to the vertices, is an isomorphism of ${\mathcal G}$.
Note that this action is not only an isomorphism of the underlying graph of ${\mathcal G}$ but an isomorphism of the embedding, thus also preserving faces.

Archdeacon's result is as follows:

\begin{theorem}\label{Archdeacon} {\rm \cite{A}} Suppose there exists a Heffter array $H(m, n; s, t)$ with orderings $\omega_r$ on the
entries in the rows of the array and $\omega_c$ on the entries in the columns of the array, where $\omega_r$ and $\omega_c$ are both simple and compatible.
Then there exists a face $2$-colourable biembedding ${\mathcal G}$ of $K_{2ms+1}$ on an orientable surface such that
the faces of one colour are cycles of length $s$ and the faces of the other colour are cycles of length $t$.
Moreover, $(\mathbb{Z}_{2ms+1},+)$ has a sharply vertex-transitive action on ${\mathcal G}$.
\end{theorem}

A corollary of this result is that the decompositions ${\mathcal C}$ and ${\mathcal C}^{\prime}$ of the graph $K_{2ms+1}$ into $s$-cycles and $t$-cycles (respectively) are orthogonal. If we relax the condition of simplicity in the above theorem, we still have a biembedding on an orientable surface but the faces collapse into smaller ones (and the cycles become circuits).
On the other hand if we relax only the condition of compatibility, we have an embedding onto a pseudosurface rather than a surface, but ${\mathcal C}$ and ${\mathcal C}^{\prime}$ remain orthogonal.

In  \cite{DM} it was verified that there exists an  $H(m,n;n,m)$ which admits both simple and compatible orderings,  for all $n\geq 3$ when  $m=3$, and for all $3\leq n\leq 100$ when $m=5$.
In \cite{ADDY, DW, CMPP} it was verified that there exists an integer $H(n;k)$,  $n\geq k$ and $nk\equiv 3\md{4}$, that admit both simple and compatible orderings,  when $k\in \{3,5,7,9\}$ and simple orderings when
 $k\in \{4,6,8,10\}$.
 Focusing on simple orderings only, the authors of \cite{BCDY} constructed simple Heffter arrays, $H(n;k)$, satisfying the conditions:
(a) $k \equiv 0 \md{4}$; or (b) $k \equiv 3\md{4}$ and $n \equiv 1 \md{4}$; or (c) $k \equiv 3 \md{4}$, $n \equiv 0 \md{4}$, and $n \gg k$.
In the current paper, we 
 establish existence results for Heffter arrays $H(n;k)$ with simple and compatible orderings where  $n \equiv 1 \md{4}$, $k \equiv 3\md{4}$  and $n$ is prime or; $n\gg k$ and either $n\not\equiv 0\md{3}$ or $p\equiv 1\md{3}$.


The starting point for our study is the following result providing necessary conditions for the existence of compatible  Heffter arrays. It is a generalization of results given in \cite{DM,CMPP,CDP}.

\begin{theorem}
If there exist compatible orderings  $\omega_r$ and $\omega_c$ for a Heffter array $H(m,n;s,t)$, then
either:
\begin{itemize}
\item $m$, $n$, $s$ and $t$ are all odd;
\item $m$ is odd, $n$ is even and $s$ is even; or
\item  $m$ is even, $n$ is odd and $t$ is even.
\end{itemize}
\label{compatt}
\end{theorem}

\begin{proof}
Let $\omega_r$ and $\omega_c$ be compatible orderings for a Heffter array $H(m,n;s,t)$.
A permutation is {\em odd} (parity 1) or {\em even} (parity 0) if it can be written as a product of {\em odd} or {\em even} transpositions, respectively. To be clear we say this is the {\em parity} of the permutation. If a permutation is a cycle of even length it has odd parity, and vice versa.

It follows that the parity of $\omega_r$ is equal to $m(s-1)$ (mod $2$) and the parity of
$\omega_c$ is equal to $n(t-1)$ (mod $2$). Thus the parity of $\omega_r\circ \omega_c$ is equal to
 $m(s-1)+n(t-1)$ (mod $2$).
But the parity of a cycle of length $ms$ is equal to $ms-1$ (mod $2$). So if the orderings are compatible,
$m(s-1)+n(t-1)-ms+1=n(t-1)-(m-1)$  is even.

Hence if $n(t-1)$ is odd, then $n$ is odd, $t$ is even and $m$ is even.
Otherwise $n(t-1)$ is even and $m$ is odd. If $n$ is odd, $t$ is odd, thus since $ms=nt$, $s$ is also odd.
Otherwise $n$ is even and $m$ is odd.  Since $ms=nt$, $s$ is even.
\end{proof}

Thus by Theorem \ref{SquHefSpec} and \ref{compatt}, if there exists an integer $H(n;k)$ with both compatible and simple orderings, then $nk\equiv 3\md{4}$. In other words either $n\equiv 1\md{4}$ and $k\equiv 3\md{4}$ or; $n\equiv 3\md{4}$ and $k\equiv 1\md{4}$.
In this context, we will verify the following theorem and show existence of Heffter arrays $H(n;k)$ with simple and compatible orderings for $n\equiv 1\md{4}$ and $k\equiv 3\md{4}$, with infinite sporadic exceptions. The case $n\equiv 3\md{4}$ and $k\equiv 1\md{4}$ remains unsolved in general.

The next four theorems are the main results of this paper. 
Theorem \ref{r1} is proven in Section \ref{Pmain2}.

\begin{theorem}\label{r1}
  Let $n\equiv 1 \md{4}$, $p>0$ and $n> 4p+3$. If there exists $\alpha$ such that
  $2p+2\leq \alpha\leq n-2-2p$, $\mbox{\rm gcd}(n,\alpha)=1$, $\mbox{\rm gcd}(n,\alpha-2p-1)=1$ and $\mbox{\rm gcd}(n,n-1-\alpha-2p)=1$,
 then there exists a globally simple integer Heffter array $H(n;4p+3)$ with an ordering that is both simple and compatible.
\end{theorem}

We then show that under certain conditions a suitable $\alpha$ exists  and prove the following theorem in Subsection \ref{Pmain3}.
\begin{theorem}\label{r2}
Let $n\equiv 1 \md{4}$, $p>0$, $n> 4p+3$ and either: (a) $n$ is prime; (b) $n=4p+5$; or (c) $n>(2p+2)^2$ and either $n\not\equiv 0\md{3}$ or $p\equiv 1\md{3}$.
 Then there exists a globally simple integer Heffter array $H(n;4p+3)$ with an ordering that is both simple and compatible.
Furthermore [by Theorem $\ref{Archdeacon}$], there exists a face $2$-colourable embedding ${\mathcal G}$ of $K_{2n(4p+3)+1}$ on an orientable surface such that
the faces of each color are cycles of length $4p+3$.
Moreover, $\mathbb{Z}_{2n(4p+3)+1}$ has a sharply vertex-transitive action on ${\mathcal G}$.
\end{theorem}

We also establish a lower bound on the number of non-isomorphic biembeddings of cycle systems from equivalent Heffter arrays. 
Let ${\mathcal G}$ 
be a biembedding of two cycle decompositions of the complete graph on an orientable surface corresponding to a Heffter array $H=H(m,n;s,t)$.
It is clear that rearranging the rows and columns of $H$ simply changes the ordering of the cycle system equivalence classes and thus has no effect on ${\mathcal G}$.
Replacing every entry $x$ in a Heffter array by $-x$ has the effect of changing the direction of each directed edge in the corresponding embedding ${\mathcal G}$.
Finally, taking the transpose of a Heffter array $H$ is equivalent to swapping the colours of the faces in the corresponding embedding ${\mathcal G}$.
It is clear that for each of these equivalences, the associated orderings for $H$ can be adjusted accordingly.
In summary, two Heffter arrays $H$ and $H^{\prime}$ are said to be {\em equivalent} if one can be obtained from the other by (i) rearranging rows or columns; (ii) replacing every entry $x$ in $H$ with $-x$; and/or (iii) taking the transpose.

 The number of non-equivalent Heffter arrays $H(n,4p+3)$ that satisfy Theorem \ref{main2} is discussed in Section \ref{count2}.
Let $\mathcal{H}(n)$ represent the number of derangements on $[n]$. It is a well-known asymptotic result that
 $\mathcal{H}(n)\sim n!/e$.

\begin{theorem}
Let $n\equiv 1 \md{4}$, $p>0$ and $n\geq 4p+3$. If there exists $\alpha$ such that
  $2p+2\leq \alpha\leq n-2-2p$, $\mbox{\rm gcd}(n,\alpha)=1$, $\mbox{\rm gcd}(n,\alpha-2p-1)=1$ and
$\mbox{\rm gcd}(n,n-1-\alpha-2p)=1$,
 then there exists at least
$(n-2)(\mathcal{H}(p-2))^2\sim (n-2)[(p-2)!/e]^2$
non-equivalent globally simple integer Heffter arrays $H(n;4p+3)$, each with an ordering that is both simple and compatible.
\end{theorem}

Finally in Section \ref{count} we determine a lower bound on the number of non-isomorphic face $2$-colorable biembeddings of $K_{2nk+1}$ on orientable surfaces where the faces are of length $k$.

\begin{theorem} Let $n\equiv 1 \md{4}$, $p>0$, $n> k=4p+3$ and either: (a) $n$ is prime; (b) $n=4p+5$; or (c) $n>(2p+2)^2$ and either $n\not\equiv 0\md{3}$ or $p\equiv 1\md{3}$. Then there exists at least
$(n-2)[(p-2)!/e]^2/nk$ non-isomorphic
  face $2$-colourable biembeddings of $K_{2nk+1}$ on an orientable surface such that
the faces are of length $k$, each with
 a sharply vertex-transitive action
of $\mathbb{Z}_{2nk+1}$.
\end{theorem}

 Throughout this paper the set of integers $\{0,1,\dots ,n-1\}$ is denoted by $[n]$ and the rows and columns of an $m\times n$ array will be indexed by $[m]$ and $[n]$, respectively.

\section{$H(n;4p+3)$ with simple and compatible orderings}\label{Pmain2}

First we provide a proof of Theorem \ref{main2}.

\begin{theorem}
Let $n\equiv 1 \md{4}$, $p>0$ and $n> 4p+3$. If there exists $\alpha$ such that
  $2p+2\leq \alpha\leq n-2-2p$, $\mbox{\rm gcd}(n,\alpha)=1$, $\mbox{\rm gcd}(n,\alpha-2p-1)=1$ and
$\mbox{\rm gcd}(n,n-1-\alpha-2p)=1$,
 then there exists a globally simple integer Heffter array $H(n;4p+3)$ with orderings that are both simple and compatible.
\label{main2}
\end{theorem}

Since the existence (via construction) of the globally simple Heffter arrays, $H(n;4p+3)$ is given in
\cite{BCDY}, it suffices to show that these Heffter arrays have orderings which are compatible and simple. To get this result we will apply Lemma \ref{compat-diag} which a generalization of results given in \cite{DM,CDP,CMPP}.

In what follows, for a partially filled array $A=[A(i,j)]$  we use $A(i,j)$ to denote the  entry  in cell $(i,j)$ of array $A$.
The cells of an $n\times n$ array can be partitioned into $n$ disjoint {\em diagonals} $D_d$, $d\in [n]$, where \begin{eqnarray*}
D_d:=\{(i+d,i)\mid i\in [n]\}.
\end{eqnarray*}
We will use the convention that if $\alpha$ and $\beta$ are two permutations acting on a set $X$, then $(\alpha\circ \beta)(x)$ is defined to be $\beta(\alpha(x))$, for each $x\in X$.

\begin{lemma}\label{compat-diag}
Assume that $k$ is odd and  that the non-empty cells of a Heffter array $H(n;k)$, $A=[A(i,j)]$ can be partitioned into diagonals $D_{g(1)},\dots,D_{g(k)}$, where $g(1)<g(2)<\dots <g(k)$.
 For $h=2,\dots, k$ define gaps of the diagonals as $s_h=g(h)-g(h-1)\md{n}$ and $s_1=g(1)-g(k)\md{n}$.
Suppose that for all $h=1,\dots,k$, {\rm gcd}$(n,s_h)=1$.
Then if $A$ is globally simple, the array $A$ has an ordering which is both simple and compatible.
\end{lemma}

\begin{proof}
We define an ordering for the Heffter array in terms of the natural orderings of each row (left to right) and column (top to bottom).
Let $\alpha_r=\phi_r$ for each row $r$, $r\in [n-1]$ and let
$\alpha_{n-1}=\phi_{n-1}^{-1}$, where $\phi_r$ is the natural ordering for each row $r\in [n]$.
For each column $c$, $c\in [n]$, let $\alpha_c=\phi_c$, where $\phi_c$ is the natural ordering for column $c$.

Observe that each $\alpha_r$ and $\alpha_c$ defined above are simple by definition.
Next, define $\omega_r$ and $\omega_c$ to be compositions of the orderings $\alpha_r$, $r\in [n]$ and $\alpha_c$, $c\in [n]$, respectively.
It remains to show that $\omega_r$ and $\omega_c$ are compatible orderings, that is $w_r\circ w_c$ can be written as a single permutation of length $nk$.
While we have defined compatible orderings based on entries above, such orderings can also be defined on the cells of an array.

To this end define mappings, $\Omega_r$ and $\Omega_c$, on the non-empty cells of $A$ as follows
\begin{eqnarray*}
\Omega_r((i,j))&=&(i,j^\prime)\mbox{ iff }\omega_r(A(i,j))=A(i,j^\prime),\mbox{ and}\\
\Omega_c((i,j))&=&(i^\prime,j)\mbox{ iff }\omega_c(A(i,j))=A(i^\prime,j).
\end{eqnarray*}
Then $\omega_r\circ \omega_c$ can be written as a single cycle if and only if $\Omega_r\circ \Omega_c$ can be written as a single cycle.
For simplicity, we will abuse notation and remove brackets writing $\Omega_r(i,j)$  instead of $\Omega_r((i,j))$; similarly for $\Omega_c(i,j)$.

For fixed $h$, consider the diagonals $D_{g(h)}, D_{g(h+1)},  D_{g(h+2)}$ and cell $(n-1,n-1-g(h))\in D_{g(h)}$.
Then working modulo $n$ on the row and column indices,
\begin{eqnarray}
(\Omega_r\circ \Omega_c)(n-1,n-1-g(h))&=&\Omega_c(n-1, n-1-g(h+1))\nonumber \\
&=&(s_{h+2},n-1-g(h+1)),\nonumber\\
\therefore (\Omega_r\circ \Omega_c)^{2}(n-1,n-1-g(h))&=&( 2s_{h+2}-1,n-g(h+1)+s_{h+2}-1),\nonumber\\
\therefore (\Omega_r\circ \Omega_c)^i(n-1,n-1-g(h))&=&(is_{h+2}-1, n-g(h+1)+(i-1)is_{h+2}-1),  1\leq i\leq n,  \nonumber\\
\therefore (\Omega_r\circ \Omega_c)^{n}(n-1,n-1-g(h))&=&(n-1, n-1-g(h+2)).\nonumber
\end{eqnarray}
Now since $k$ is odd and each $s_{h}$ is coprime to $n$, we see that mapping $\Omega_r\circ \Omega_c$ is a full cycle of length $nk$.
\end{proof}

\begin{theorem} {\rm \cite{BCDY}}\label{4p+3}
Let $n\equiv 1 \md{4}$, $p>0$ and $n\geq 4p+3$. If there exists $\alpha$ such that {\rm gcd}$(\alpha,n)=1$ and
$2p+2\leq \alpha\leq n-2-2p$,
then
 there exists a globally simple Heffter array $H(n;4p+3)$, denoted by $B$, with occupied cells on the set of diagonals
$$D_0,D_1,\dots,D_{4p-2},D_{2p+\alpha-3},D_{2p+\alpha-1},D_{2p+\alpha},D_{2p+\alpha+1}.$$
\end{theorem}

Now observe that the gaps between the diagonals for the Heffter array in Theorem \ref{4p+3} are of size either $1$, $2p+\alpha-3-(4p-2)=\alpha-2p-1$, $2$ or $n-(2p+\alpha+1)$. So Lemma \ref{compat-diag} together with Theorem \ref{4p+3} then imply Theorem \ref{main2}.

\subsection{Existence of a suitable $\alpha$} \label{Pmain3}

In this section we will give some lemmata using number theory to determine when a suitable $\alpha$ exists for use in Theorem \ref{main2} to prove Theorem \ref{main3}.

\begin{theorem}\label{main3}
Let $n\equiv 1 \md{4}$, $p>0$, $n> 4p+3$ and either: (a) $n$ is prime; (b) $n=4p+5$; or (c) $n>(2p+2)^2$ and; either $n\not\equiv 0\md{3}$ or $p\equiv 1\md{3}$.
 Then there exists a globally simple integer Heffter array $H(n;4p+3)$ with an ordering that is both simple and compatible.
Furthermore [by Theorem $\ref{Archdeacon}$], there exists a face $2$-colourable embedding ${\mathcal G}$ of $K_{2n(4p+3)+1}$ on an orientable surface such that
the faces of each color are cycles of length $4p+3$.
Moreover, $\mathbb{Z}_{2n(4p+3)+1}$ has a sharply vertex-transitive action on ${\mathcal G}$.
\end{theorem}

\begin{lemma}\label{1}
If $n\equiv 0\md{3}$ and $p\not\equiv 1\md{3}$, there does not exist $\alpha<n$ such that $\mbox{\rm gcd}(n,\alpha)=1$,
$\mbox{\rm gcd}(n,\alpha-2p-1) =1$ and
$\mbox{\rm gcd}(n,n-(\alpha+2p+1))=1$.
\end{lemma}

\begin{proof}
Let $n\equiv 0\md{3}$ and $p\not\equiv 1\md{3}$. Suppose there exists an
$\alpha$ that satisfies all three of the gcd conditions.
Then $\mbox{\rm gcd}(n,\alpha)=1$  implies $\alpha\not\equiv 0\md{3}$, hence we have two options $\alpha\equiv 1$ or $2 \md{3}$. Since $p\not\equiv 1 \md{3}$, there are further two options to consider: $p\equiv 0$ or $2\md{3}$. Thus there are four cases to consider in all. However each of these cases leads to a contradiction.
\end{proof}

\begin{lemma}\label{2}
If $n=4p+5$, $\alpha=2p+2$ satisfies the conditions
 $2p+2\leq \alpha\leq n-2-2p$,
$\mbox{\rm gcd}(n,\alpha)=1$,
$\mbox{\rm gcd}(n,\alpha-2p-1)=1$ and
$\mbox{\rm gcd}(n,n-1-\alpha-2p)=1$.
\end{lemma}

\begin{lemma}\label{3}
Let $n\geq (2p+2)^2$ and $n$ be odd. Further if $n\equiv 3\md{6}$ then $p\equiv 1\md{3}$. Then there exists $\alpha$ satisfying
$2p+2\leq \alpha\leq n-2-2p$,
$\mbox{\rm gcd}(n,\alpha)=1$, $\mbox{\rm gcd}(n,\alpha-2p-1)=1$ and
$\mbox{\rm gcd}(n,n-1-\alpha-2p)=1$.
\end{lemma}

\begin{proof}
The proof is trivial if $n$ is prime. Otherwise let $q_1<q_2<\dots <q_h$ be the prime factors of $n$ where $h\geq 2$.
For each $i$, there exists $0<b_i<q_i$ such that $b_i-2p-1\not\equiv 0\md{q_i}$ and
$-1-b_i-2p\not\equiv 0\md{q_i}$. (Note that if $q_1=3$, we need $p\equiv 1\md{3}$ here for $b_1$ to exist.)
By the Chinese remainder theorem, there is a unique $x$ satisfying
$x\equiv b_i\md{q_i}$ for each $1\leq i\leq h$ and $0<x< q_1q_2\dots q_h$.
Observe that if $2p+2\leq x\leq n-2-2p$, then $\alpha=x$ has the required properties and we are done.

Otherwise we need to make some adjustments to $x$.
Let $Q=q_1q_2\dots q_h$.
Suppose there is a prime $q$ such that $q^2$ divides $n$.
Since $n\geq (2p+2)^2$, we have that $(n-2-2p)-(2p+2)=n-4p-4>n/3\geq n/q\geq Q$, so there exists
$\alpha\equiv x\md{Q}$ such that $2p+2\leq \alpha\leq n-2-2p$ and we are done.

Otherwise $n=Q$.
Let $b_1^{\prime}\neq b_1$ satisfy the same properties as $b_1$ above (note that such a $b_1^{\prime}$ exists even if $q_1=3$).
 By the Chinese remainder thorem, there is a unique $x^{\prime}$ satisfying
$x^{\prime}\equiv b_1^{\prime}\md{q_1}$, $x^{\prime}\equiv b_i\md{q_i}$ for each $2\leq i\leq h$ and $0<x^{\prime}< q_1q_2\dots q_h$.
Moreover, $x\equiv x^{\prime} \md{n/q_1}$ and without loss of generality $x-x^{\prime}\geq n/q_1$.
Since $q_1$ is the least prime that divides composite $n$, $n/q_1>\sqrt{n}$.
Since $n\geq (2p+2)^2$, we thus have that $x-x^{\prime}> 2p+2$. It follows that at least one of $\alpha=x$ or $\alpha=x^{\prime}$ satisfies
 $2p+2\leq \alpha\leq n-2-2p$.
\end{proof}

Then Theorem \ref{main3} follows directly from Theorem \ref{main2} and previous Lemmas.

\section{Non-equivalent globally simple Heffter arrays, $H(n;4p+3)$}\label{Pmain4}

In this section we work towards proving Theorem \ref{main4}.

\begin{theorem}
Let $n\equiv 1 \md{4}$, $p>0$ and $n\geq 4p+3$. If there exists $\alpha$ such that
  $2p+2\leq \alpha\leq n-2-2p$, $\mbox{\rm gcd}(n,\alpha)=1$, $\mbox{\rm gcd}(n,\alpha-2p-1)=1$ and
$\mbox{\rm gcd}(n,n-1-\alpha-2p)=1$,
 then there exists at least
$(n-2)(\mathcal{H}(p-2))^2\sim (n-2)[(p-2)!/e]^2$
non-equivalent globally simple integer Heffter arrays $H(n;4p+3)$, each with an ordering that is both simple and compatible.
\label{main4}
\end{theorem}

We start with a generalization of Heffter arrays. An array $A$ is defined to be a {\em support shifted simple integer Heffter array} $H(n;4p,\gamma)$ if it satisfies the following properties:
\begin{itemize}
\item[{\bf P1.}] Every row and every column of $A$ has $4p$ filled cells.
\item[{\bf P2.}] $s(A)=\{\gamma n+1,\dots,(4p+\gamma)n\}$.
\item[{\bf P3.}] Elements in every row and every column sum to $0$.
\item[{\bf P4.}] Partial sums are distinct in each row and each column of $A$  modulo $2(4p+\gamma)n+1$.
\end{itemize}

A related generalization of Heffter arrays is studied in \cite{ItalE2}. Note that a support shifted integer Heffter array $H(n;4p,0)$ is in fact an integer Heffter array $H(n;4p)$. Support shifted simple integer Heffter arrays were constructed for all $n\geq 4p$ and $\gamma>1$ in \cite{BCDY}. Then these arrays for $\gamma=3$ were merged with a Heffter array $H(n;3)$ to obtain simple Heffter arrays $H(n;4p+3)$. In this section we first document the existing constructions from \cite{BCDY} then we will generalize these constructions to obtain $(p-1)!(p-2)!$ non-equivalent support shifted simple $H(n;4p,\gamma)$. Then as in \cite{BCDY} we will merge each of these arrays with Heffter arrays $H(n;3)$ to prove Theorem \ref{main4}.

\subsection{Existing results on support-shifted simple integer Heffter arrays, $H(n;4p,\gamma)$}
First, we outline the precise results needed from \cite{BCDY}.

For an $n\times n$ array let the entries in  row $a$ and column $a$ of diagonal $D_i$ be denoted by $d_i(r_a)$ and $d_i(c_a)$ respectively, with these values defined to be $0$ when there is no entry.
For a given row $a$ we define  $\Sigma(x)=\sum_{i=0}^x d_i(r_a)$ and
for a given column $a$ we define  $\overline{\Sigma}(x)=\sum_{i=0}^x d_i(c_a)$.
For a given row $a$, the values of $\Sigma(x)$ such that $d_x(r_a)$ is non-zero are called the {\em row partial sums} for $a$.
For a given column $a$, the values of $\overline{\Sigma}(x)$ such that $d_x(c_a)$ is non-zero are called the {\em column partial sums} for $a$.
Thus to show that a Heffter array $H(n;k)$ is globally simple, it suffices to show that the row partial sums are distinct (modulo $2nk+1$) for each row $a$ and that the column partial sums are distinct (modulo $2nk+1$) for each column $c$.
It is important to be aware that throughout this section, row and column indices are {\em always} calculated modulo $n$, while entries of arrays are {\em always} evaluated as integers.

\begin{remark} It will be useful to refer to the following basic observations. Let $m$, $x_1$, $x_2$, $\alpha_1$, $\alpha_2$, $\beta_1$, $\beta_2$ be integers and $m>0$. Then for:
 \begin{eqnarray}
 -m\leq x_1,x_2\leq m,&& x_1\equiv x_2\md{2m+1}\Rightarrow x_1=x_2;  \label{mods}\\
 0\leq x_1,x_2< m,&& x_1\equiv x_2\md{m} \Rightarrow x_1=x_2; \label{modl}\\
 -\frac{m}{2}<\alpha_1,\alpha_2<\frac{m}{2},&& \beta_1m+\alpha_1=\beta_2m+\alpha_2\Rightarrow \beta_1=\beta_2\ and\ \alpha_1=\alpha_2; \label{n}\\
 -m<x_1<0<x_2<m,&& x_1\equiv x_2\md{m}\Rightarrow x_2=m+x_1.  \label{mod=}
\end{eqnarray}
\end{remark}

In \cite{BCDY} a globally simple array $A$ was constructed as follows.
Let $\gamma>0$,  $n\geq 4p$, $2p-1\leq \alpha\leq n-2p-1$, and gcd$(\alpha,n)=1$. Define $I=[p]$, $J=[p-1]$ and $A=[A(i,j)]$ to be the $n\times n$ array with filled cells defined by the $4p$ diagonals
$$D_{2i}, D_{2i+1}, D_{2p}, D_{2p+1+2j}, D_{2p+2+2j},D_{2p+\alpha},$$
where $ i\in I$ and $ j\in J$, and with entries for each $x\in [n]$:
\begin{eqnarray}
(\gamma+2)n+4in-2x&\mbox{ in cell}&(2i-x,-x)\in  D_{2i}\nonumber,\nonumber\\
-\gamma n-4in-1-2x &\mbox{ in cell}&(2i+1+x,x)\in D_{2i+1},\nonumber \\
-(4p+\gamma)n+2x&\mbox{ in cell}&(2p-\alpha x,-\alpha x)\in  D_{2p},\label{A}\\
(4p+\gamma-6)n-4jn+1+2x&\mbox{ in cell}&(2p+1+2j-x,-x)\in  D_{2p+1+2j},\nonumber\\
-(4p+\gamma-4)n+4jn+2x&\mbox{ in cell}&(2p+2+2j+x,x)\in  D_{2p+2+2j},\nonumber\\
(4p+\gamma-2)n+1+2x&\mbox{ in cell}&(2p+\alpha+\alpha x,\alpha x)\in  D_{2p+\alpha}\nonumber.
\end{eqnarray}

\begin{theorem} {\rm (Theorem 3.1 of \cite{BCDY})}\label{ccddyy}
Let $\gamma>0$,  $n\geq 4p$, $2p-1\leq \alpha\leq n-2p-1$, and {\rm gcd}$(\alpha,n)=1$. Then the array $A$ constructed above is a support-shifted simple integer Heffter array $H(n;4p,\gamma)$.%
%
\end{theorem}

\begin{remark}
Choose $\alpha=2p-1$ when constructing $A$, then $2p-1\leq \alpha\leq n-2p-1$ and {\rm gcd}$(\alpha,n)=1$. Hence, if $\gamma>0$ and $n\geq 4p$, then there exists a support shifted simple integer Heffter array $H(n;4p,\gamma)$.
\end{remark}

From Equations (12) and (13) in \cite{BCDY}, we have the following lemma.
\begin{lemma} \label{partsum}{\rm \cite{BCDY}}
 The row partial sums and the column partial sums of $A$ satisfy the following inequalities.
\begin{eqnarray}
\Sigma(4p-2)<\Sigma(4p-4)<\dots<\Sigma(2p+2)<\Sigma(2p)<-(4p+\gamma-3)n<0 \nonumber\\
0< \Sigma(1) <\Sigma(3)<\dots<\Sigma(2p-1)<\gamma n.  \label{Rowparsum4p}
\end{eqnarray}\begin{eqnarray}
-(4p+\gamma+1)n<\overline{\Sigma}(2p)<\overline{\Sigma}(2p+2)<\dots<\overline{\Sigma}(4p-2)<\overline{\Sigma}(2p)+p<-n\nonumber\\
 -n<\overline{\Sigma}(2p-1)<\dots<\overline{\Sigma}(3)<\overline{\Sigma}(1)<0 .\label{Colparsum4p}
  \end{eqnarray}
\end{lemma}

\subsection{The existence of many non-equivalent support-shifted simple integer Heffter arrays, $H(n;4p,\gamma)$}

In this section we reorder the entries in each column of the array $A$ given in the previous section to get a new array $A^\prime$, obtained by
applying a  bijection $f_I:I\rightarrow I$ to the entries in the coupled diagonals $D_{2i}$ and $D_{2i+1}$ of $A$ and a bijection $f_J:J\rightarrow J$  to the entries in the coupled diagonals $D_{2p+2j+1}$ and $D_{2p+2j+2}$ of $A$.

Let $\gamma>0$,  $n\geq 4p$, $2p-1\leq \alpha\leq n-2p-1$, and gcd$(\alpha,n)=1$. For each pair of functions
$(f_I,f_J)$, we construct an $n \times n$ array
$A^\prime$ with support $s(A^\prime)=\{\gamma n+1,\dots,(4p+\gamma)n\}$  as follows: for all $i\in I$, $j\in J$,  and $x\in [n]$  in $A^\prime$ place entry
\begin{eqnarray}\label{Aprime}
(\gamma+2)n+4f_I(i)n-2x&\mbox{ in cell}&(2i-x,-x)\in  D_{2i}\nonumber,\nonumber\\
-\gamma n-4f_I(i)n-1-2x &\mbox{ in cell}&(2i+1+x,x)\in D_{2i+1},\nonumber \\
-(4p+\gamma)n+2x&\mbox{ in cell}&(2p-\alpha x,-\alpha x)\in  D_{2p},\nonumber\\
(4p+\gamma-6)n-4f_J(j)n+1+2x&\mbox{ in cell}&(2p+1+2j-x,-x)\in  D_{2p+1+2j},\nonumber\\
-(4p+\gamma-4)n+4f_J(j)n+2x&\mbox{ in cell}&(2p+2+2j+x,x)\in  D_{2p+2+2j},\nonumber\\
(4p+\gamma-2)n+1+2x&\mbox{ in cell}&(2p+\alpha+\alpha x,\alpha x)\in  D_{2p+\alpha}.
\end{eqnarray}

We illustrate this new construction with an example.

\begin{example}
Here $n=17$, $p=3$, $\alpha=2p=(2\times 3)=6$, $f_{I}(0)=0$, $f_{I}(1)=2$, $f_{I}(2)=1$, $f_{J}(0)=1$ and $f_{J}(1)=0$.
\begin{scriptsize}
\begin{center}
\begin{tabular}{|c|c|c|c|c|c|c|c|c|c|c|c|c|c|c|c|c|}
\hline
85&&&&&252&&-173&172&-101&100&-253&-144&145&-216&217&-84\\
\hline
-52&53&&&&&224&&-171&170&-99&98&-225&-146&147&-218&219\\
\hline
221&-54&55&&&&&230&&-169&168&-97&96&-231&-148&149&-220\\
\hline
-188&189&-56&57&&&&&236&&-167&166&-95&94&-237&-150&151\\
\hline
153&-190&191&-58&59&&&&&242&&-165&164&-93&92&-243&-152\\
\hline
-120&121&-192&193&-60&61&&&&&248&&-163&162&-91&90&-249\\
\hline
-255&-122&123&-194&195&-62&63&&&&&254&&-161&160&-89&88\\
\hline
86&-227&-124&125&-196&197&-64&65&&&&&226&&-159&158&-87\\
\hline
-119&118&-233&-126&127&-198&199&-66&67&&&&&232&&-157&156\\
\hline
154&-153&116&-239&-128&129&-200&201&-68&69&&&&&238&&-155\\
\hline
-187&186&-115&114&-245&-130&131&-202&203&-70&71&&&&&244&\\
\hline
&-185&184&-113&112&-251&-132&133&-204&205&-72&73&&&&&250\\
\hline
222&&-183&182&-111&110&-223&-134&135&-206&207&-74&75&&&&\\
\hline
&228&&-181&180&-109&108&-229&-136&137&-208&209&-76&77&&&\\
\hline
&&234&&-179&178&-107&106&-235&-138&139&-210&211&-78&79&&\\
\hline
&&&240&&-177&176&-105&104&-241&-140&141&-212&213&-80&81&\\
\hline
&&&&246&&-175&174&-103&102&-247&-142&143&-214&215&-82&83\\
\hline
\end{tabular}
\end{center}
\end{scriptsize}
\end{example}

\begin{theorem} Let $p>0$, $n\geq 4p$, $(n,\alpha)=1$, $2p-1\leq \alpha\leq n-2p-1$ and $\gamma> 0$. Then there exists at least $(p-1)!(p-2)!$ non-equivalent support shifted simple integer Heffter arrays $H(n;4p,\gamma)$ where filled cells are precisely the set of diagonals $\{D_1,D_2,\dots,D_{4p-2},D_{2p+\alpha}\}$ .
\end{theorem}

We  prove this theorem by showing that for each pair of functions $(f_I,f_J)$, the array $A^\prime$ constructed above carries the Properties {\rm P1, P2, P3, P4}.
We will use the notation $\Sigma_A(x)$ and $\overline{\Sigma}_A(x)$ to denote the row partial sums and column partial sums in the array $A$ as given in \cite{BCDY} and $\Sigma_{A^\prime}(x)$ and $\overline{\Sigma}_{A^\prime}(x)$ to denote the  row partial sums and column partial sums respectively in the array $A^\prime$ as constructed here. Observe that $A$ is a special form of the array $A^\prime$ where both $f_I$ and $f_J$ are identity mappings.

Now since $A^\prime$ is obtained by permuting the entries in columns of $A$, $s(A^\prime)=\{\gamma n+1,\dots, (4p+\gamma)n\}$ and all columns sums are equal to $0$. Next the equations in Lemma \ref{Sums} can be used to verify that the row sums are 0.

\begin{lemma}\label{Sums}
The rows and columns $A^{\prime}$ satisfy the following equations for each $i\in I$ and $j\in J$. (Note, that since the context is clear here we have reduced notation and represented $d_{x}(r_a)$ and $d_x(c_a)$ by $d_x$.)

\begin{tabular}{lll}
For rows $a$:&For columns $a\neq 0$:&For columns $a=0$:\\
$d_{2i}+d_{2i+1}=1,$&$d_{2i}+d_{2i+1}=-1$&$d_{2i}+d_{2i+1}=2n-1,$\\\label{iSum}
$d_{2p+2j+1}+d_{2p+2j+2}=-1,$&$d_{2p+2j+1}+d_{2p+2j+2}=1,$&$d_{2p+2j+1}+d_{2p+2j+2}=-2n+1,$\\\label{jSum}
$d_{2p}+d_{2p+\alpha}=-1,$&$d_{2p}+d_{2p+\alpha}= 1,$&$d_{2p}+d_{2p+\alpha}=-2n+1.$\label{AlphaSum}\\
\end{tabular}
\end{lemma}

\begin{proof}
In \cite{BCDY} it was shown the above statements were  true for the array $A$ and the  result follows directly  by  definition for the columns of $A^{\prime}$.

For a given row $a\in [n]$ and all $i\in I$, there exists $x_1,x_2\in [n]$ such that $a=2i-x_1\md{n}$ and $a=2i+1+x_2\md{n}$. Thus  $x_1+x_2+1=0\md{n}$  and so $x_1+x_2=n-1.$ Consequently for all $ i\in I$,
\begin{eqnarray}
d_{2i}(r_a)+d_{2i+1}(r_a)&=&(\gamma+2)n+4f_I(i)n-2x_1-\gamma n-4f_I(i)n-1-2x_2\nonumber\\
&=&2n-1-2(n-1)=1.\label{RowiSum}
\end{eqnarray}
The remaining observations for the rows hold similarly.
\end{proof}

\begin{corollary}\label{cdy2}
 For any choice of $f_I$ and $f_J$ and for all rows and columns of $A^{\prime}$ we have:
   \begin{eqnarray*}
   &\Sigma_{A^\prime}(2i+1)= \Sigma_{A}(2i+1),  &\overline{\Sigma}_{A^\prime}(2i+1)= \overline{\Sigma}_{A}(2i+1), \\
   &\Sigma_{A^\prime}(2p+2j+2)=\Sigma_{A}(2p+2j+2), &\overline{\Sigma}_{A^\prime}(2p+2j+2)=\overline{\Sigma}_{A}(2p+2j+2),\\
   &\Sigma_{A^\prime}(2p)=\Sigma_{A}(2p), & \overline{\Sigma}_{A^\prime}(2p)=\overline{\Sigma}_{A}(2p),\\
   &\Sigma_{A^\prime}(2p+\alpha)=\Sigma_{A}(2p+\alpha), &\overline{\Sigma}_{A^\prime}(2p+\alpha)=\overline{\Sigma}_{A}(2p+\alpha).
   \end{eqnarray*}
\end{corollary}

We will also need the following lemma to bound certain partial sums.
 \begin{lemma}
The following bounds hold for partial sums on rows and non-zero columns.
\begin{eqnarray*}
-(4p+\gamma)n & < \Sigma_{A^{\prime}}(2p)=&d_{2p}(r_a)+p <   -(4p+\gamma-2)n+p-1, \\
& \Sigma_{A^{\prime}}(4p-2)= & d_{2p}(r_a)+1, \\
  -(4p+\gamma)n-p & \leq \overline{\Sigma}_{A^{\prime}}(2p)=&d_{2p}(c_a)-p \leq  -(4p+\gamma-2)n-p-2.
\end{eqnarray*}
\label{bbound}
\end{lemma}
\begin{proof}
From Lemma \ref{Sums},  for each row $a$, $\Sigma_{A^{\prime}}(2p)=d_{2p}(r_a)+p$ and for each non-zero column $a$,
 $\overline{\Sigma}_{A^{\prime}}(2p)=d_{2p}(c_a)-p$. Also,  $\Sigma_{A^{\prime}}(4p-2)=\Sigma_{A^{\prime}}(2p) - (p-1)$.

 The result then follows from the definition of $A^{\prime}$.
\end{proof}

   By Theorem \ref{ccddyy}, $A$ has distinct row and column partial sums; hence by Corollary \ref{cdy2} it is only necessary to show that
   the row partial sums $\Sigma_{A^\prime}(2i)$ and $\Sigma_{A^\prime}(2p+2j+1)$ and column partial sums $\overline{\Sigma}_{A^\prime}(2i)$ and $\overline{\Sigma}_{A^\prime}(2p+2j+1)$ are distinct from each other and from the other partial sums.

  By Lemma \ref{Sums} and the definition of $A^\prime$, for all rows $a$
\begin{eqnarray*}
   \Sigma_{A^\prime}(2i)&=&d_{2i}(r_a)+i=(\gamma+2)n+4f_I(i)n-2x+i\nonumber\\
   && \mbox{ where }x=2i-a\md{n}\mbox{ and }0\leq x\leq n-1 \mbox{, and}\\
   \Sigma_{A^\prime}(2p+2j+1)&=&\Sigma_{A^\prime}(2p)+d_{2p+2j+1}(r_a)-j\\
   &=&\Sigma_{A^\prime}(2p)+(4p+\gamma-6)n-4f_J(j)n+1+2x-j\mbox{,} \\
   &&\mbox{ where }\ x=2p+2j+1-a\md{n}\mbox{ and }0\leq x\leq n-1\mbox{.}
\end{eqnarray*}
Note that
\begin{eqnarray}
(4p+\gamma)n>\Sigma_{A^\prime}(2i)\geq \gamma n+2, \label{Range2i}\\
0>-n>\Sigma_{A^\prime}(2p+2j+1)>\Sigma_{A^\prime}(2p)+(\gamma+1)n=d_{2p}(r_a)+p+(\gamma+1)n, \label{Range2j+1}
\end{eqnarray}
since $0\leq f_I(i)\leq p-1$ and $0\leq f_J(j)\leq p-2$ for $i\in I$ and $j\in J$.

Furthermore $\Sigma_{A^\prime}(2i_1)\equiv \Sigma_{A^\prime}(2i_2) \md{2(4p+\gamma)n+1}$ for some $i_1,i_2\in I$ implies $$4f_I(i_1)n-2(2i_1-a\md{n})+i_1=4f_I(i_2)n-2(2i_2-a\md{n})+i_2$$ by (\ref{Range2i}) and (\ref{mods}).
Then $4(f_I(i_1)-f_I(i_2))n\leq (i_2-i_1)+2(n-1)\leq 3n$ which implies $f_I(i_1)-f_I(i_2)=0$. Hence $i_1=i_2$.
Similarly $\Sigma_{A^\prime}(2p+2j_1+1)\equiv \Sigma_{A^\prime}(2p+2j_2+1)\md{2(4p+\gamma)n+1}$ for some $j_1,j_2\in J$ implies $$4f_I(j_1)n-2(2p+2j_1+1-a\md{n})+j_1=4f_I(j_2)n-2(2p+2j_2+1-a\md{n})+j_2$$ by (\ref{Range2j+1}) and (\ref{mods}). Similarly this implies $4(f_I(j_1)-f_I(j_2))n\leq (j_2-j_1)+2(n-1)\leq 3n$ and hence $j_1=j_2$ as before. Therefore by inequality (\ref{Rowparsum4p}) all row partial sums are distinct.

Next, let $a\neq 0$ be a column.
By Lemma \ref{Sums} and the definition of $A^\prime$ we have:
\begin{eqnarray*}
 \overline{\Sigma}_{A^\prime}(2i) & = & d_{2i}(c_a)-i=(\gamma+2)n+4f_I(i)n-2(n-a\md{n})-i,  \\
  \overline{\Sigma}_{A^\prime}(2p+2j+1)&=&\overline{\Sigma}_{A^\prime}(2p)+d_{2p+2j+1}(c_a)+j=d_{2p}(c_a)-p+d_{2p+2j+1}(c_a)+j \\
  &<&-(4p+\gamma)n+2n-2-p+(4p+\gamma-6)n-4f_J(j)n+1+2n-2+j\nonumber\\
  &=&-2n-3-p-4f_J(j)n+j < -n,\nonumber\\
 \overline{\Sigma}_{A^\prime}(2p+2j+1) &\geq&d_{2p}(c_a)-p+(4p+\gamma-6)n-4f_J(j)n+1+j\nonumber\\
  &>&\overline{\Sigma}_{A^\prime}(2p)+(\gamma+2)n.
\end{eqnarray*}
Thus,
\begin{eqnarray}
 (4p+\gamma)n &>& \overline{\Sigma}_{A^\prime}(2i) \geq\gamma n+2-i>0, \label{Col2i}\\
  \overline{\Sigma}_{A^\prime}(2p-1)  & > & \overline{\Sigma}_{A^\prime}(2p+2j+1) > \overline{\Sigma}_{A^\prime}(2p)+(\gamma+2)n> \overline{\Sigma}_{A^\prime}(2p)+p. \label{Col2j+1}
\end{eqnarray}

Furthermore $\overline{\Sigma}_{A^\prime}(2i_1)\equiv\overline{\Sigma}_{A^\prime}(2i_2) \md{2(4p+\gamma)n+1}$ implies $4f_I(i_1)n-i_1=4f_I(i_2)n-i_2$ by (\ref{mods}). Then $4(f_I(i_1)-4f_I(i_2))n=i_1-i_2$ so $i_1=i_2$.
Similarly $\overline{\Sigma}_{A^\prime}(2p+2j_1+1)=\overline{\Sigma}_{A^\prime}(2p+2j_2+1)\mod{(2(4p+\gamma)n+1)}$ implies $-4f_I(j_1)n+j_1=-4f_I(j_2)n+j_2$ so $j_1=j_2$ as before. Therefore by inequality (\ref{Colparsum4p}) all column partial sums are distinct.

  Now consider column $0$. Then:
  \begin{eqnarray}
  \overline{\Sigma}_{A^\prime}(2i)&=&(2n-1)i+(\gamma+2)n+4f_I(i)n\leq 2(4p+\gamma)n+1,\label{zeroi} \\
  \overline{\Sigma}_{A^\prime}(2p+2j+1)&=& (p-j)(2n-1)-(4p+\gamma)n +(4p+\gamma-6)n-4f_J(j)n+1\label{zeroj}\\
  &=&(p-j)(2n-1)-6n-4f_J(j)n+1>-(4p+\gamma)n.\nonumber
 \end{eqnarray}

 By Lemma  \ref{Sums} and Corollary \ref{cdy2}, related partial sums for column $0$ can be calculated as:
\begin{eqnarray}
 \overline{\Sigma}_{A^\prime}(2i+1)&=&2(i+1)n-(i+1)>0,\nonumber \\
 \overline{\Sigma}_{A^\prime}(2p)&=&-(2p+\gamma)n-p<0,\nonumber \\
\overline{\Sigma}_{A^\prime}(2p+2j+2)&=&-(2p+2j+\gamma+2)n-(p-j-1)<0, \label{04p-2} \\
\overline{\Sigma}_{A^\prime}(2p+\alpha)&=&0. \nonumber
  \end{eqnarray}

 Observe that for column $0$ and for all non-empty diagonals $x$, \begin{eqnarray}|\overline{\Sigma}_{A^\prime}(x)|\leq 2(4p+\gamma)n+1.\label{star4}\end{eqnarray}
  We will show that for each $i\in I$ and $j\in J$, $\overline{\Sigma}_{A^\prime}(2i)$ and $\overline{\Sigma}_{A^\prime}(2p+2j+1)$ are distinct $\md{2(4p+\gamma)n+1}$ from each other and each of the other partial sums in column $0$. In what follows we will make extensive use of
(\ref{star4}) together with (\ref{n}).
\begin{itemize}
\item[1(i)] Suppose that $\overline{\Sigma}(2i_1)=\overline{\Sigma}(2i_2)\md{2(4p+\gamma)n+1}$ for some $i_1,i_2\in I$. Then $2(i_1-i_2)n-(i_1-i_2)=4(f_I(i_2)-f(i_1))n$ by (\ref{modl}) but then $i_1-i_2=0$.
\item[1(ii)] Suppose that $\overline{\Sigma}(2i_1)=\overline{\Sigma}(2i_2+1)\md{2(4p+\gamma)n+1}$ for some $i_1,i_2\in I$. Then  (\ref{modl}) implies
    $$(2n-1)i_1+(\gamma+2)n+4f_I(i_1)n= (2i_2+2)n-(i_2+1).$$ Hence $i_1=i_2+1$ and $2i_1+\gamma+2+4f_I(i_1)=2i_2+2$. But then $4f_I(i_1)=-\gamma-2$. This is a contradiction since $\gamma> 0$.

\item[1(iii)] Suppose that $\overline{\Sigma}(2i)=\overline{\Sigma}(2p)\md{2(4p+\gamma )n+1}$, for some $i\in I$. Then by (\ref{mod=})
\begin{eqnarray*}
(2n-1)i+(\gamma+2)n+4f_I(i)n&=&2(4p+\gamma)n+1-(2p+\gamma)n-p,\\
(\gamma+2+4f_I(i)+2i)n-i&=&(6p+\gamma)n+1-p.
\end{eqnarray*} This implies $i=p-1$ and $\gamma+2+4f_I(i)+2i=6p+\gamma$ leading to the contradiction $f_I(i)=p$.
\item[1(iv)] Suppose that $\overline{\Sigma}(2i)=\overline{\Sigma}(2p+2j+1)\md{2(4p+\gamma )n+1}$, for some $i\in I$ and $j\in J$ then by (\ref{mod=})
\begin{eqnarray*}
(2n-1)i+(\gamma+2)n+4f_I(i)n &=& (p-j)(2n-1)-6n-4f_J(j)n+1\quad  \mbox{ or} \\
(2n-1)i+(\gamma+2)n+4f_I(i)n &=& (p-j)(2n-1)-6n-4f_J(j)n+2(4p+\gamma)n+2.
 \end{eqnarray*}
The former case implies $i=p-j-1$ and so $2(p-j-1)+\gamma+2+4f_I(i)=2p-2j-6-4f_J(j)$, or equivalently $4(f_I(i)+f_J(j))=-6-\gamma<0$ which is a contradiction. For the second case we have $(2i+2j+8-\gamma+4f_I(i)+4f_J(j))n-i=10pn-(p-j)+2$, which implies $p-2=i+j$ and so $4(f_I(i)+f_J(j))=10p-2(i+j)-8+\gamma=10p-2p-4+\gamma=8p-4+\gamma>8p-4$. This is a contradiction since $f_I(i)+f_J(j)\leq 2p-3$ and so $4(f_I(i)+f_J(j))\leq 8p-12$.

\item[1(v)]  Suppose that $\overline{\Sigma}(2i)=\overline{\Sigma}(2p+2j+2)\md{2(4p+\gamma)n+1}$ for some $i\in I$ and $j\in J$. Then by (\ref{mod=})
\begin{eqnarray*}
(2n-1)i+(\gamma+2)n+4f_I(i)n&=& 2(4p+\gamma)n+1 -(2p+2j+\gamma+2)n-(p-j-1)\\
&=&(6p-2j+\gamma-2)n-(p-j-2).
\end{eqnarray*}
 Thus  $i=p-j-2$ and $2i+\gamma+2+4f_I(i)=6p-2j+\gamma-2$ or equivalently $2(p-j-2)+4f_I(i)+2=6p-2j-2$ and so  $f_I(i)=p$, a contradiction.

\item[2(i)]
 Assume $\overline{\Sigma}(2p+2j+1)\equiv \overline{\Sigma}(2i+1)\md{2(4p+\gamma )n+1}$  for some $i\in I$ and $j\in J$. Then by (\ref{mods}) we have
 $$(p-j)(2n-1)-6n-4f_J(j)n+1=(2i+2)n-(i+1).$$ Then $p=j+i+2$ and $2p-2j-6-4f_J(j)-2i-2=0$ which implies $2(i+j+2)-2i-2j-8-4f_J(j)=0$ and so $f_J(j)=-1$, a contradiction.

\item[2(ii)]  Assume $\overline{\Sigma}(2p+2j+1)\equiv \overline{\Sigma}(2p) \md{2(4p+\gamma )n+1}$ for some $j\in J$. Then by (\ref{mods})
$$(p-j)(2n-1)-6n-4f_J(j)n+1= -(2p+\gamma)n-p.$$ This implies $-p=-p+j+1$ which is a contradiction as $j\neq -1$.

\item[2(iii)] Assume $\overline{\Sigma}(2p+2j_1+1)\equiv \overline{\Sigma}(2p+2j_2+2)\md{2(4p+\gamma )n+1}$ for some $j_1, j_2\in J$. Then by (\ref{mods})
    $$(p-j_1)(2n-1)-6n-4f_J(j_1)n+1=-(2p+2j_2+\gamma+2)n-(p-j_2-1).$$ Then we have $-(p-j_1)+1=-p+j_2+1$ and $4p-2j_1-6-4f_J(j_1)+2j_2+\gamma+2=0$ which implies $j_1=j_2$ and $4f_J(j_1)=4p-4+\gamma\geq 4(p-1)$. This is a contradiction since $f_J(j_1)\leq p-2$ .

\item[2(iv)] Assume $\overline{\Sigma}(2p+2j+1)\equiv 0\md{2(4p+\gamma)n+1}$ for some $j\in J$. Then by (\ref{mods}) we have $-p+j+1=0$ which implies $j=p-1>p-2$, a contradiction.

\item[2(v)] Assume $\overline{\Sigma}(2p+2j_1+1)\equiv \overline{\Sigma}(2p+2j_2+1)\md{2(4p+\gamma)n+1}$. Then by (\ref{mods}) we have $-p+j_1+1=-p+j_2+1$ which implies $j_1=j_2$.
\end{itemize}

\subsection{Non-equivalent globally simple Heffter arrays $H(n;4p+3)$ }\label{count2}

In this section we prove Theorem \ref{main4}.
First we need the following theorems from \cite{BCDY}.

\begin{theorem} {\rm \cite{BCDY}} For each $n\equiv 1\md{4}$ and $0\leq \beta\leq n-5$, a Heffter array $H(n;3)$, $L$, exists that satisfies the following conditions:
\begin{itemize}
\item The non-empty cells are exactly on the diagonals $D_{\beta}$, $D_{\beta+2}$ and $D_{\beta+4}$,
\item $s(D_{\beta+2})=\{1,\dots,n\}$,
\item $s(D_\beta\cup D_{\beta+4})=\{n+1,\dots,3n\}$,
\item entries on $D_{\beta}$ are all positive,
\item entries on $D_{\beta+4}$ are all negative,
\item the array defined by $M=[M(i,j)]$ where $M(i,j)=L(i+1,j+1)$, $i,j\in [n]$ retains the above properties.
\end{itemize}
\label{ladder}
\end{theorem}

\begin{theorem}\label{4p+3-2} {\rm \cite{BCDY}}
Let $n\equiv 1\md{4}$, $p>0$ and $n> 4p+3$. Let $\alpha$ be an integer such that $2p+2\leq \alpha\leq n-2-2p$ and $\mbox{\rm gcd}(n,\alpha)=1$.
Let $\beta=2p+\alpha-3$ and let $L$ be a Heffter array $H(n;3)$ based on $\beta$ satisfying the properties of Theorem
$\ref{ladder}$ where $\{L(\beta,0),-L(\beta+4,0)\}\cap\{2n-1,2n-(2p+1)/3\}=\emptyset$. Then
the union of arrays $L$ and $A$ (defined in  \eqref{A}, with $\gamma=3$) is a globally simple Heffter array $H(n;4p+3)$ where entries are on the set of diagonals $D_i$ where $i$ is in
${\mathbb D}=\{0,1,\dots ,4p-2,2p+\alpha\}\cup \{2p+\alpha-3,2p+\alpha-1,2p+\alpha+1\}$.
\end{theorem}
 As $A^{\prime}$ satisfies properties P1, P2, P3 and P4, and is indeed on the same set of cells as $A$, a similar theorem can be proven if we replace $A$ with $A^{\prime}$ under certain conditions on $f_I$ and $f_J$ in $A^\prime$:
$f_I(0)= 0$, $f_I(i)\neq (2p-i+1)/2$ for $i\in I$ and $f_J(j)\neq (p-j-4)/4$ for $j\in J$. Furthermore assume $\{L(\beta,0),-L(\beta+4,0)\}\cap\{2n-1\}=\emptyset$.

Let $B^{\prime}$ be the  merged array constructed as below.
\begin{eqnarray*}
B^\prime(i,j)=
\left\{\begin{array}{ll}
 A^\prime(i,j)\mbox{ if }i-j\not\in\{2p+\alpha-3,2p+\alpha-1,2p+\alpha+1\},\\
 L(i,j)\mbox{ if }i-j\in\{2p+\alpha-3,2p+\alpha-1,2p+\alpha+1\}.
 \end{array}\right.
\end{eqnarray*}
Hence we are positioning $D_{2p+\alpha-3}$, $D_{2p+\alpha-1}$ and $D_{2p+\alpha+1}$ of $L$ to match the diagonals $D_{2p+\alpha-3}$, $D_{2p+\alpha-1}$ and $D_{2p+\alpha+1}$ of $B^\prime$.

\begin{theorem}\label{Bprime} Let $n\equiv 1\md{4}$, $p>0$, $n> 4p+3$ and $\alpha$ be an integer such that $(n,\alpha)=1$ and $2p+2\leq \alpha\leq n-2-2p$. Now assume  $f_I(0)= 0$, $f_I(i)\neq (2p-i+1)/2$ for $i\in I$ and $f_J(j)\neq (p-j-4)/4$ for $j\in J$ and $\{L(\beta,0),-L(\beta+4,0)\}\cap\{2n-1\}=\emptyset$. Then the array  $B^\prime$ constructed as above by merging arrays $A^\prime$ and $L$ is a globally simple Heftter array $H(n;4p+3)$.
\end{theorem}

\begin{proof}
Now,
$s(A^\prime)=\{3n+1,\dots, (4p+3)n\}$ and $s(L)=\{1,\dots,3n\}$ so it is obvious that $s(B)=\{1,\dots,(4p+3)n\}$. Also since all row and column sums of both $A^\prime$ and $L$ are $0$, all rows and columns sum to $0$ in $B^\prime$.

 For rows and columns of $B^{\prime}$, $\Sigma_{B^\prime}(x)= \Sigma_{A^\prime}(x)$ and $\overline{\Sigma}_{B^\prime}(x)= \overline{\Sigma}_{A^\prime}(x)$ for $0\leq x\leq 4p-2$. Hence by the previous section the row partial sums and column partial sums are distinct for $0\leq x\leq 4p-2$.

 Also $\Sigma_{B^\prime}(x)= \Sigma_{B}(x)$ and $\overline{\Sigma}_{B^\prime}(x)= \overline{\Sigma}_{B}(x)$ for all $x\in\{2i+1,2p+2j+2,2p,2p+\alpha-3, 2p+\alpha-1, 2p+\alpha, 2p+\alpha+1|i\in I, j\in J\}$  where $B$ is the array given in Theorem \ref{4p+3}, in effect when $f_I$ and $f_J$ are identity mappings. Hence to prove that row and non-zero columns have distinct partial sums in $B^\prime$, we only need to show that when $i\in I$ and $j\in J$
 \begin{eqnarray*}
\{ \Sigma_{B^\prime}(2i),\Sigma_{B^\prime}(2p+2j+1)\}\cap \Sigma_{B^\prime}(2p+\alpha-3),\Sigma_{B^\prime}(2p+\alpha-1),\Sigma_{B^\prime}(2p+\alpha)\}=\emptyset\mbox{ and}\\
 \{ \overline{\Sigma}_{B^\prime}(2i),\overline{\Sigma}_{B^\prime}(2p+2j+1)\}\cap \overline{\Sigma}_{B^\prime}(2p+\alpha-3),\overline{\Sigma}_{B^\prime}(2p+\alpha-1),\overline{\Sigma}_{B^\prime}(2p+\alpha)\}=\emptyset,
\end{eqnarray*}
where the above elements are calculated as residues modulo $2(4p+3)n+1$.

First note that \begin{eqnarray} n+1\leq\Sigma_{B^{\prime}}(2p+\alpha), \overline{\Sigma}_{B^{\prime}}(2p+\alpha)   \leq 3n. \label{RangeL}\end{eqnarray}.

   Now consider row $a$, since $\Sigma_{B^{\prime}}(4p-2)=\Sigma_{A^{\prime}}(4p-2)=d_{2p}(r_a)+1$ (from Lemma \ref{bbound}),
  the row partial sums for array $B^\prime$ are as follows:
   \begin{eqnarray}
\Sigma_{B^\prime}(2p+\alpha-3)&=&d_{2p}(r_a)+1+L(a,a-2p-\alpha+3)<0,\nonumber\\
\Sigma_{B^\prime}(2p+\alpha-1)&=&d_{2p}(r_a)+1-L(a,a-2p-\alpha-1)<0,\nonumber\\
\Sigma_{B^\prime}(2p+\alpha)&=&-L(a,a-2p-\alpha-1)>0,\nonumber\\
\Sigma_{B^\prime}(2p+\alpha+1)&=&0.\nonumber
\end{eqnarray}

   Then by (\ref{Range2i}) and (\ref{Range2j+1}) we have:
  $(4p+3)n>\Sigma_{B^\prime}(2i)\geq 3n+2>\Sigma_{B^\prime}(2p+\alpha)>n>0>\Sigma_{B^\prime}(2p+2j+1)> d_{2p}(r_a)+3n+2>\Sigma_{B^\prime}(2p+\alpha-3),\Sigma_{B^\prime}(2p+\alpha-1)>d_{2p}(r_a) \geq -(4p+3)n.$
The final inequality can be inferred from the definition of $A^{\prime}$.

  Next, consider column $a\neq 0$. Since $f_I(0)=0$, from (\ref{Col2i}), we may deduce that $\overline{\Sigma}_{B^\prime}(2i)\geq 3n+1$.
Furthermore, by (\ref{Col2i}) and (\ref{Col2j+1}) we have:

   $(4p+3)n>\overline{\Sigma}_{B^\prime}(2i)\geq 3n+1>\overline{\Sigma}_{B^\prime}(2p+\alpha)>n>0>\overline{\Sigma}_{B^\prime}(2p+2j+1)> d_{2p}(c_a)+3n+2>\overline{\Sigma}_{B^\prime}(2p+\alpha-3),\overline{\Sigma}_{B^\prime}(2p+\alpha-1)>-(4p+3)n$.

  Now consider column $a=0$. By (\ref{zeroi}) and (\ref{zeroj}) we have:
  \begin{eqnarray}
  \overline{\Sigma}_{B^\prime}(2i)&=&(2n-1)i+5n+4f_I(i)n\leq 2(4p+3)n+1,\nonumber \\
 \overline{\Sigma}_{B^\prime}(2p+2j+1)&=
  &(p-j)(2n-1)-6n-4f_J(j)n+1>-(4p+3)n.\nonumber
 \end{eqnarray}

Since $\overline{\Sigma}_{B^{\prime}}(4p-2)=-(4p+1)n-1$, we also have:
\begin{eqnarray}
\overline{\Sigma}_{B^{\prime}}(2p+\alpha-3)&=&\overline{\Sigma}_{B}(2p+\alpha-3)=-(4p+1)n-1+L(2p+\alpha-3,0)<0,\label{seec1} \\
\overline{\Sigma}_{B^{\prime}}(2p+\alpha-1)&=&\overline{\Sigma}_{B}(2p+\alpha-1)=-(4p+1)n-1-L(2p+\alpha+1,0)<0,\label{seec2} \\
\overline{\Sigma}_{B^{\prime}}(2p+\alpha)&=&\overline{\Sigma}_{B}(2p+\alpha)=-L(2p+\alpha+1,0)>0,\nonumber\\
\overline{\Sigma}_{B^{\prime}}(2p+\alpha+1)&=&0\nonumber.
\end{eqnarray}

 Assuming $\{L(\beta,0),-L(\beta+4,0)\}\cap\{2n-1\}=\emptyset$, it was shown in \cite{BCDY} (Proof of Theorem 3.10) that partial sums $\overline{\Sigma}_{B}(2p+\alpha-3)$, $\overline{\Sigma}_{B}(2p+\alpha-1)$ and $\overline{\Sigma}_{B}(2p+\alpha)$ are distinct from partial sums $\overline{\Sigma}_{A}(2i+1)$, $\overline{\Sigma}_{A}(2p)$ and $\overline{\Sigma}_{A}(2p+2j+2)$ for all $i\in I$ and $j\in J$ and; $0$. So by Corollary \ref{cdy2} we only need to show that partial sums $\overline{\Sigma}_{B^{\prime}}(2p+\alpha-3)$, $\overline{\Sigma}_{B^{\prime}}(2p+\alpha-1)$ and $\overline{\Sigma}_{B^{\prime}}(2p+\alpha)$ are also distinct from each of the other partial sums $\overline{\Sigma}_{B^\prime}(2i)$ and $\overline{\Sigma}_{B^\prime}(2p+2j+1)$ for all $i\in I$ and $j\in J$.

 In what follows we will make extensive use of (\ref{star4}) together with (\ref{n}).
\begin{itemize}
\item[1)(i)]   $\overline{\Sigma}_{B^{\prime}}(2i)=(2n-1)i+5n+4f_I(i)n\geq 5n$ so by (\ref{RangeL}) and (\ref{modl}),
$$\overline{\Sigma}_{B^{\prime}}(2p+\alpha)\not\equiv \overline{\Sigma}_{B^{\prime}}(2i)\md{2(4p+3)n+1}$$  for all $i\in I$.

\item[1)(ii)]  Suppose that $\overline{\Sigma}_{B^{\prime}}(2i)\equiv \overline{\Sigma}_{B^{\prime}}(2p+\alpha-3)\md{2(4p+3)n+1}$ for some $i\in I$.  Then by (\ref{mod=}),
$\overline{\Sigma}_{B^{\prime}}(2i)=(2n-1)i+5n+4f_I(i)n=2(4p+3)n+1 -(4p+1)n-1+L(2p+\alpha-3,0)$.

Hence $ (4f_I(i)+2i-4p)n-i=L(2p+\alpha-3,0)$ which implies $4f_I(i)+2i-4p=2$ and so $f_I(i)=(2p-i+1)/2$, contradicting the definition of $f_I$.

\item[2)(i)] Suppose that $\overline{\Sigma}_{B^{\prime}}(2p+2j+1)\equiv \overline{\Sigma}_{B^{\prime}}(2p+\alpha)\md{2(4p+3)n+1}$ for some $j\in J$. Then $$(p-j)(2n-1)-6n-4f_J(j)n+1= -L(2p+\alpha+1,0)$$ which implies $2p-2j-6-4f_J(j)=2$  by (\ref{RangeL}). Then $4f_J(j)=2p-2j-8$ so $f_J(j)=(p-j-4)/2$, contradicting the definition of $f_J$.

\item[2)(ii)] Suppose that $\overline{\Sigma}_{B^{\prime}}(2p+2j+1)\equiv \Sigma(2p+\alpha-3)\md{2(4p+3)n+1}$ for some $j\in J$.  Then $$(p-j)(2n-1)-6n-4f_J(j)n+1= -(4p+1)n-1+L(2p+\alpha-3,0)$$ which implies  $L(2p+\alpha-3,0)=(6p-5-2j-4f_J(j))n-p+j+2$. Now by (\ref{RangeL}) $6p-5-2j-4f_J(j)\in \{1,2,3\}$.
Thus $f_J(j)\geq p+(p-j-3)/2\geq p-1$, a contradiction.

\end{itemize}

Finally, the above can similarly be verified for  $\overline{\Sigma}_{B^{\prime}}(2p+\alpha-1)$ since this is bounded by  the same range of values as $\overline{\Sigma}_{B^{\prime}}(2p+\alpha-3)$, from (\ref{seec1}) and (\ref{seec2}).

This completes the proof of Theorem \ref{Bprime}.
\end{proof}

It remains to prove Theorem \ref{main4}. Assume $n\equiv 1\md{4}$, $p>0$, $n> 4p+3$ and $\alpha$ be an integer such that $2p+2\leq \alpha\leq n-2-2p$, $\mbox{\rm gcd}(n,\alpha)=1$, $\mbox{\rm gcd}(n,\alpha-2p+2)=1$ and
$\mbox{\rm gcd}(n,n-1-\alpha-2p)=1$. Let $B^{\prime}$ and $B^{\prime\prime}$ be two Heffter arrays constructed as in the previous theorem for distinct valid choices of $f_I$ and $f_J$.
Firstly, from the definition of $A^\prime$, the diagonals $D_{2p}$ and $D_{2p+\alpha}$ do not depend on the choice of $f_I$ and $f_J$.
Then since these two diagonals of $B^\prime$ and $B^{\prime\prime}$ are identical yet other entries are different, it is impossible to obtain $B^{\prime\prime}$ from $B^{\prime}$ by transpose, rearranging rows and/or columns, or replacing each entry $x$ with $-x$. Therefore, $B^{\prime}$ and $B^{\prime\prime}$ are non-equivalent Heffter arrays. Also by Lemma \ref{compat-diag}  they have simple and compatible orderings.

Next, recall the restrictions $f_I(0)= 0$, $f_I(i)\neq (2p-i+1)/2$ for $i\in I$ and $f_J(j)\neq (p-j-4)/4$ for $j\in J$.
  Thus, if $\mathcal{H}(\epsilon)$ represents the number of derangements on a set of size $\epsilon$ then there are more than $\mathcal{H}(p-2)$ choices for each of $f_I$ and $f_J$. Moreover, from the final condition of Theorem \ref{ladder}, the
subarray $L$ of $B^{\prime}$ may be adjusted provided
$\{L(\beta,0),-L(\beta+4,0)\}\cap\{2n-1\}=\emptyset$. This yields an extra factor of $n-2$ in the number of Heffter arrays.

\section{A lower bound on the number of non-isomorphic biembeddings of $(4p+3)$-cycle systems on orientable surfaces}\label{count}


Let ${\mathcal G}$ and ${\mathcal G}^{\prime}$ be biembeddings of a two cycle decomposition of the complete graph on an orientable surface corresponding to Heffter arrays $H=H(m,n;s,t)$ and $H^{\prime}$ as given in Theorem \ref{Archdeacon}, respectively.
We consider the possibility that ${\mathcal G}$ and ${\mathcal G}^{\prime}$ are isomorphic embeddings while $H$ and $H^{\prime}$ are non-equivalent Heffter arrays. Let $f$ be such an isomorphism acting on the vertex set of ${\mathcal G}$.

Let $C$ be a cycle of length $m$ corresponding to a black face  in ${\mathcal G}$. Since the isomorphism is incidence preserving there exists a $C^{\prime}$ of length $n$ which corresponds to a black face in ${\mathcal G}^{\prime}$. There are
$n(2nt+1)$ choices for $C^{\prime}$ and for a specific vertex $x$ incident with $C$ there are initially $t$ choices for $f(x)$ in $C^{\prime}$ (from above, we can assume that edge direction is preserved). Thus once  $C^{\prime}$ and $f(x)$ are chosen the rest of the mapping $f$ is forced.
Thus, since there is a sharply vertex-transitive automorphism of ${\mathcal G}^{\prime}$, there are at most $nt$ choices for the isomorphism $f$.

Consequently, Theorem \ref{main4} together with Lemmas \ref{1}, \ref{2} and \ref{3} will imply the following.

\begin{theorem} Let $n\equiv 1 \md{4}$, $p>0$, $n> k=4p+3$ and either: (a) $n$ is prime; (b) $n=4p+5$; or (c) $n>(2p+2)^2$ and either $n\not\equiv 0\md{3}$ or
$p\equiv 1\md{3}$. Then there exists at least
$(n-2)[(p-2)!/e]^2/nk$ non-isomorphic
  face $2$-colourable biembeddings of $K_{2nk+1}$ on an orientable surface such that
the faces are of length $k$, each with
 a sharply vertex-transitive action
of $\mathbb{Z}_{2nk+1}$.
\end{theorem}

\vspace{5mm}
\noindent{\bf Acknowledgment:} The fourth author would like to acknowledge support from
TUBITAK 2219 and the School of Mathematics and Physics, The University of Queensland, through the awarding of a Ethel Raybould Visiting Fellowship.

 \end{document}